\documentclass[12pt]{amsart}
\usepackage{amsmath,amssymb,amsbsy,amsfonts,amsthm,latexsym,amsopn,amstext,amsxtra,epic, euscript,amscd,indentfirst}
\usepackage{latexsym,epsfig,subfigure,afterpage,graphics,epsfig,color}
\begin{document}

\newtheorem{theorem}{Theorem}
\newtheorem{conjecture}[theorem]{Conjecture}
\newtheorem{proposition}[theorem]{Proposition}
\newtheorem{question}[theorem]{Question}
\newtheorem{lemma}[theorem]{Lemma}
\newtheorem{cor}[theorem]{Corollary}
\newtheorem{obs}[theorem]{Observation}
\newtheorem{proc}[theorem]{Procedure}
\newcommand{\comments}[1]{} 
\def\Z{\mathbb Z}
\def\Za{\mathbb Z^\ast}
\def\Fq{{\mathbb F}_q}
\def\R{\mathbb R}
\def\N{\mathbb N}
\def\cH{\overline{\mathcal H}}
\def\cF{\mathcal F}
\def\ip{i_p}
\def\lp{l_p}
\def\dap{d_{a,p^2}}
\def\Im{{\rm Image}(d_{a,p^2})}

\title[Combinatorial Geometry and Modular Hyperbolas]{Two Combinatorial Geometric Problems involving Modular Hyperbolas}

\author[Khan]{Mizan R. Khan}
\address{MRK: Department of Mathematics and Computer Science, Eastern Connecticut State University, Willimantic, CT 06226, USA}
\email{khanm@easternct.edu}
\author[Magner]{Richard Magner}
\address{RM: Department of Mathematics and Computer Science, Eastern Connecticut State University, Willimantic, CT 06226}
\email{magnerri@my.easternct.edu}
\author[Senger]{Steven Senger}
\address{SS: Department of Mathematical Sciences, University of Delaware, Newark, DE 19716, USA}
\email{senger@math.udel.edu}
\author[Winterhof]{Arne Winterhof}
\address{AW: Johann Radon Institute for Computational and Applied Mathematics, Austrian Academy of Sciences, 
Altenberger Str.\ 69, A-4040 Linz,Austria}
\email{arne.winterhof@oeaw.ac.at}

\date{\today}

\begin{abstract}
For integers $a$ and $n\ge 2$ with $\gcd(a,n)=1$ let $\cH_{a,n}$ be the set of least residues of a modular hyperbola
$$ \cH_{a,n} = \{ (x,y) \in \Z^2 \, : \, xy \equiv a \!\!\!\! \pmod{n}, 1 \le x,y \le n-1 \}.$$ 
In this paper we prove two combinatorial geometric results about $\cH_{a,p^m}$, where $p^m$ is a prime power. 
Our first result shows that the number of ordinary lines spanned by $\cH_{1,p^m}$ is at least
$$(p-1)p^{m-1}\left(\frac{p^{m-1}(p-2)}{2} +c(p^m)\right),$$ 
where  
\begin{enumerate}
\item[(i)] $c(p)=0$, $c(4)=1/2$, $c(8)=0$, $c(49)=6/7$;
\item[(ii)] $c(p^m)=6/13$ if $m\ge 2$, $p^m$ is small and $p^m\not= 4,8,49$;
\item[(iii)] $c(2^m)=1/2$ if $m$ is sufficiently large;
\item[(iv)] $c(p^m)=3/4+o(1)$ if $p>2$, $m\ge 2$ and $p^m$ is sufficiently large.
\end{enumerate}
In the special case of $m=1$ we have equality.  

The second result gives a partial answer 
to a question of Shparlinski~\cite{IS} on the cardinality of  
$$\cF_{a,n}= \{ \sqrt{x^2 + y^2} \, : \, (x,y) \in \cH_{a,n} \}.$$ 
\end{abstract}

\maketitle 

\section{Ordinary lines in $\cH_{1,p^m}$}

For $n \ge 2$ let $\Za_n$ be the group of invertible elements modulo $n$ and let ${\mathcal H}_{a,n}$ denote the {\em modular hyperbola} $xy \equiv a \pmod{n}$ where 
$x,y,a \in \Z$, with $\gcd(a,n)=1$. (We insert the condition that $a$ and $n$ are relatively prime to ensure that  ${\mathcal H}_{a,n} \subseteq \Za_n \times \Za_n$.) 
Following~\cite{IS} we define  
$\cH_{a,n} = {\mathcal H}_{a,n} \cap [1,n-1]^2$, that is,
$$ \cH_{a,n} = \{ (x,y) \in \Z^2 \, : \, xy \equiv a \!\!\!\! \pmod{n}, 1 \le x,y \le n-1 \}.$$ 
In the special case of $\cH_{1,n}$ we will simply drop the 1 and write $\cH_n$.

Let $S$ be a finite set of points in the Euclidean space. A line that passes through exactly two distinct points of $S$ is said to be an \emph{ordinary line} spanned by $S$. The notion of an ordinary line arose in the context of the famous \emph{Sylvester-Gallai} theorem in combinatorial geometry.

\begin{theorem}[Sylvester-Gallai] Let P be a finite set of points in the plane, not all on a line. Then there is an ordinary line spanned by P.
\end{theorem}
 
We refer the reader to~\cite{GT} and the references therein for an exposition of the history of this theorem and subsequent developments. We now give an application of the Sylvester-Gallai theorem to modular hyperbolas.

\begin{lemma}
The only moduli for which the modular hyperbolas $\cH_n$ do \emph{not} span an ordinary line are $n=2,8,12$ and $24$.
\end{lemma}

\begin{proof}
We may assume that $n\not=2,3,4,6$ since for $n=2$ the modular hyperbola consists of only one point and in the case of $n=3,4$ or $6$ 
the modular hyperbola consists of only two points and so for these $3$ cases we have precisely one ordinary line. 

The points $(1,1)$ and $(n-1,n-1)$ are two distinct points of $\cH_n$, and consequently $\cH_n$ spans the line $y=x$. We now observe that the number of solutions of the congruence $z^2 \equiv 1 \pmod{n}$ equals $\varphi(n)$ precisely when $n=2,3,4,6,8,12$ and $24$. For all other values of $n$ there exists $z \in \Z_n^\ast$ such that $z^2 \not\equiv 1 \pmod{n}$. Such a $z$ gives a point in $\cH_n$ that does not lie on $y =x$. We now invoke the Sylvester-Gallai theorem to conclude our proof.
\end{proof}

For prime moduli it is easy to determine the precise number of ordinary lines. 

\begin{lemma}\label{prime-moduli}
Let $p$ be a prime. Then the set $\cH_{a,p}$ spans $(p-1)(p-2)/2$ ordinary lines.
\end{lemma}

\begin{proof}
We show that any line connecting $2$ distinct points of $\cH_{a,p}$ is ordinary. Let $(x_1,y_1), (x_2,y_2)$ be two distinct points in $\cH_{a,p}$, and let $y=kx+d$ be the line in $\R^2$ 
passing through these two points. Since $x_1 \not= x_2$,  $x_1$ and $x_2$ are distinct roots modulo $p$ of the quadratic polynomial $kx^2+dx-a$. By Lagrange's theorem $kx^2+dx-a$ has no more than $2$ roots modulo $p$. 
Hence, no other point of $\cH_{a,p}$ lies on $y=kx+d$. Since $\#\cH_{a,p} = p-1$, $\cH_{a,p}$ spans ${p-1 \choose 2}$ ordinary lines.
\end{proof}

For the rest of this section we focus on the case $a=1$ and notice that
for prime powers $p^m$ with $m \geq 2$ (and $p^m \not= 4,9$) such a result no longer holds as $\cH_{p^m}$ spans lines that are not ordinary. In particular we have the following example.

\begin{lemma}\label{many-coll-pts}
Let $p$ be a prime and let $m \in \Z$ with $m \geq 2$ and $p^m >8$. Then $\cH_{p^m}$ spans a line with  $\left(p^{\lfloor m/2 \rfloor}-1\right)$ points.
\end{lemma}

\begin{proof}
We include the hypothesis $p^m >8$ to ensure that 
$$\left(p^{\lfloor m/2 \rfloor}-1\right) \ge 2.$$ 
Let $l$ be the line $L: x+y = p^m+2.$ We will show that 
$$\#\left(\cH_{p^m} \cap L\right) = p^{\lfloor m/2 \rfloor}-1.$$

The lattice points on the line $L$ that lie inside the first quadrant are of the form $(k,p^m+2-k)$ with 
$k=1,2,\dots, p^m,p^m+1.$ If $(k,p^m+2-k)\in \cH_{p^m}$, then 
$$k(2-k) \equiv 1 \!\!\!\! \pmod{p^m},$$
that is,  
$$(k-1)^2 \equiv 0 \!\!\!\! \pmod{p^m}.$$
Therefore, 
$$k-1 = lp^{\lceil m/2\rceil}$$ 
with $l= 1,2, \ldots, \left(p^{\lfloor m/2\rfloor}-1\right).$
\end{proof} 

Since $\left(3^{\lfloor 3/2 \rfloor}-1\right) =2$, the above proof does not show that there are non-ordinary lines for the case $p^m =27$. For this case we note that the line $x+y=38$ contains 4 points of $\cH_{27}$. 

We now state the main result of this section. 

\begin{theorem}\label{Main-result}
Let $p^m$ be a prime power and $N$ the number of ordinary lines that $\cH_{p^m}$ spans. Then
\begin{equation}\label{eq:low-bd-ord-line}
N  \geq  p^{m-1}(p-1)\left(\frac{p^{m-1}(p-2)}{2} + c(p^m)\right),
\end{equation}
where 
\begin{enumerate}
\item[(i)] $c(p)=0$, $c(4)=1/2$, $c(8)=0$, $c(49)=6/7$;
\item[(ii)] $c(p^m)=6/13$ if $m\ge 2$, $p^m$ is small and $p^m\not= 4,8,49$;
\item[(iii)] $c(2^m)=1/2$ if $m$ is sufficiently large;
\item[(iv)] $c(p^m)=3/4+o(1)$ if $p>2$, $m\ge 2$ and $p^m$ is sufficiently large.
\end{enumerate}
In the special case of $m=1$ we have equality. (We also have equality when $p^m=4,8,49$.)
\end{theorem}
 
Our proof does not include the special cases $p^m = 4,8$ and $49$. For these 3 cases we simply computed the value of $c(n)$ that gives equality. The basic idea of the proof of Theorem~\ref{Main-result} is to decompose $\cH_{p^m}$ into $p-1$ distinct subsets $C_i,i=1,\dots,p-1,$ such that any line that connects a point in $C_i$ to a point in $C_j$, with $i\not=j$, must be ordinary. We define the $C_i$ follows:
$$C_i = \left\{ (x,y) \in \cH_{p^m} \, : \, x \equiv i \!\!\!\! \pmod{p} \right\},$$ where $1 \le i \le p-1$.
We note that $\# C_i = p^{m-1}$. We now state and prove the key lemma of this section.

\begin{lemma}\label{useful1}
Let $m\geq 2$ be an integer, $p>2$ be a prime and let $L$ be a line 
$$ L: ax +by +c = 0, \textrm{ with } \gcd(a,b,c)=1,$$
that is spanned by $\cH_{p^m}$. Then we have the following:
\begin{enumerate}
\item[(i)] $c^2 - 4ab \equiv 0 \pmod{p}$ if and only if $L \cap \cH_{p^m} \subseteq  C_i$ with 
$$i = -c(2a)^{-1} \!\!\!\! \mod p.$$
\item[(ii)] $\gcd(ab,p) = 1$.
\item[(iii)] If $\#(L \cap \cH_{p^m}) \ge 3$, then $L \cap \cH_{p^m} \subseteq  C_i$ for some $i$.
\item[(iv)] If $\#(L \cap \cH_{p^m}) \ge 3$, then $\gcd(c,p) = 1$. 
\item[(v)] If $c=0$, then $L$ is the line $y=x$.
\end{enumerate}

\end{lemma}

\begin{proof}
Let $f(x)$ denote the polynomial $ax^2+cx+b$.

(i) If $c^2 - 4ab \equiv 0 \pmod{p}$ then by the quadratic formula
$$2ax \equiv -c \!\!\!\! \pmod{p}$$ for any $(x,y) \in L \cap \cH_{p^m}$. If $p|a$ then $p|c$. By considering the equation of $L$ we conclude that $p|b$. This contradicts our hypothesis that
$\gcd(a,b,c) = 1$. Thus, $\gcd(a,p)=1$ and $x \equiv -c(2a)^{-1} \pmod{p}$ for every $(x,y) \in L \cap \cH_{p^m}$.

For the opposite direction, we show that if 
$$(x_1,y_1),(x_2,y_2) \in L \cap \cH_{p^m}$$
with $x_1\ne x_2$ (and thus $y_1\ne y_2$) and $x_1 \equiv x_2 \pmod{p}$, then 
$$2ax_1 \equiv -c \!\!\!\! \pmod{p} \textrm{ and } 2by_1 \equiv -c \!\!\!\! \pmod{p}.$$ The result then
follows by multiplying these two congruences together.

We prove the first congruence. Now,
$$a(x+h)^2+c(x+h) +b = (ax^2+cx+b)+(2ax+c)h + ah^2.$$
Setting $x=x_1$ and $h=x_2-x_1$ we obtain
$$ax_2^2+cx_2+b = (ax_1^2+cx_1+b)+(2ax_1+c)(x_2-x_1)+a(x_2-x_1)^2.$$
Since $f(x_1) \equiv f(x_2) \equiv 0 \pmod{p^m}$, we get 
$$(2ax_1+c+a(x_2-x_1))(x_2-x_1) \equiv 0 \!\!\!\! \pmod{p^m}.$$
Since $x_1\equiv x_2 \pmod{p}$, but $x_1 \not\equiv x_2 \pmod{p^m}$, we infer that 
$$2ax_1+c+a(x_2-x_1) \equiv 0 \!\!\!\! \pmod{p^l}$$
for some $l,0<l<m$, and conclude that 
$$2ax_1 +c \equiv 0 \!\!\!\! \pmod{p}.$$
A similar proof yields $2by_1 \equiv - c \pmod{p}$.

(ii) We argue by contradiction. Suppose $\gcd(ab,p) = p$. Without loss of generality  we may assume that $p|a$.  Let $(x_1,y_1)$ and $(x_2,y_2)$ be two distinct points in $L \cap \cH_{p^m}$. 
Therefore $$f(x_1) \equiv f(x_2) \equiv 0 \!\!\!\! \pmod{p^m}.$$ 
Since $p|a$, $f(x)$ reduces modulo $p$ to the linear polynomial $cx +b$. Since $x_1$ and $x_2$ are both zeros of the congruence 
$cx + b \equiv 0 \pmod{p}$, we conclude that $x_1 \equiv x_2 \pmod{p}$. By part (i) we now obtain
$$-c \equiv 2ax_1 \equiv 0 \!\!\!\! \pmod{p}.$$
Since $p|a$, we have that $p|c$ and consequently $p|b$. But this contradicts our hypothesis that $\gcd(a,b,c) = 1$. 

(iii) Since $\#(L \cap \cH_{p^m}) \ge 3,$ we have by (ii) that $\gcd(ab,p)=1$. Now using the modulo $p^m$ invertible linear transformation  $z=2ax+c$ we transform the congruence $f(x) \equiv 0 \pmod{p^m}$ to 
$$ z^2\equiv (c^2-4ab) \!\!\!\! \pmod{p^m}.$$
If $\gcd(c^2-4ab,p)=1$, then the congruence $z^2\equiv (c^2-4ab) \pmod{p}$ would have precisely 2 solutions and by Hensel's lemma each one would lift to a unique solution of  $z^2\equiv (c^2-4ab)\pmod{p^m}$. This in turn would imply 
$f(x) \equiv 0 \pmod{p^m}$ has precisely 2 solutions. This contradicts our hypothesis that $f(x) \equiv 0 \pmod{p^m}$ has at least 3 solutions. Consequently, we must have that $(c^2-4ab) \equiv 0 \pmod{p}$, and the result follows from (i).

(iv) By (iii) $c^2\equiv 4ab \pmod{p}$. Since $\gcd(ab,p)=1$, $p\not|c$.

(v) By (iv) $\#(L\cap \cH_{p^m}) =2$. Let $(x_1,y_1)$ and $(x_2,y_2)$ be the two elements of the intersection 
$L\cap \cH_{p^m}$. From the fact that $x_1$ and $x_2$ are the two distinct solutions of $x^2 \equiv -ba^{-1} \pmod{p^m}$ we conclude that $x_2 = p^m-x_1$ and $y_2 = p^m-y_1$. Furthermore the slope of $L$ is
$$ -\frac{b}{a} = \frac{y_1}{x_1} = \frac{p^m-y_1}{p^m-x_1} .$$
From the last equality we get that $x_1=y_1$ and thus $b=-a$, that is,  $L$ is the line $y=x$.
\end{proof}

The next proposition gives an upper on the number of points of $\cH_{p^m}$ that can be collinear.

\begin{proposition}\label{ptline-upperbound}
Let $p>2$ be prime and $L$ a line spanned by $\cH_{p^m}$. Then
 \begin{equation}\label{eq:pt-line-up-bd}
\# (L \cap \cH_{p^m}) \leq 2 p^{\lfloor m/2 \rfloor}.
\end{equation}
\end{proposition}

\begin{proof}
Any point $(x,y) \in L \cap \cH_{p^m}$
gives rise to a solution of the quadratic congruence
$$ax^2+cx+b \equiv 0 \!\!\!\! \pmod{p^m}.$$
Thus we have that 
$$ \# (L \cap \cH_{p^m}) \leq N,$$
where $N$ denotes the number of solutions of 
\begin{equation}\label{eq:quad-cong1}
ax^2+cx+b \equiv 0 \!\!\!\! \pmod{p^m}
\end{equation}
with $0 \leq x < p^m$.

Let $D=c^2-4ab$ be the discriminant of the quadratic equation. Using the modulo $p^m$ invertible linear transformation $z=2ax+c$ we can reduce congruence~\eqref{eq:quad-cong1} to the form
\begin{equation}\label{eq:quad-cong2}
z^2\equiv D \!\!\!\! \pmod{p^m}.
\end{equation}
Hence each solution $x$ of our original congruence corresponds to a solution of this simpler congruence. 

We obtain the bound~\eqref{eq:pt-line-up-bd} by determining the solutions of congruence~\eqref{eq:quad-cong2}. 
There are three possibilities: 
\begin{enumerate}
\item[(i)] $D \not\equiv 0 \pmod{p}$; 
\item[(ii)] $D \equiv 0 \pmod{p^m}$; 
\item[(iii)] $D \equiv 0 \pmod{p}$ but $D \not\equiv 0 \pmod{p^m}.$
\end{enumerate}

(i) Since~\eqref{eq:quad-cong2} has a solution and $D\not\equiv0 \pmod{p}$, the Legendre symbol $(D/p)=1$. 
Therefore the congruence
$z^2\equiv D \pmod{p}$ has precisely 2 solutions. By Hensel's lemma each one lifts to a unique solution of  $z^2\equiv D\pmod{p^m}$. This in turn implies that $f(x) \equiv 0 \pmod{p^m}$ has precisely 2 solutions.

(ii) If $D\equiv 0 \pmod{p^m}$, then the solutions of~\eqref{eq:quad-cong2} are of the form $kp^{\lceil m/2 \rceil}$ with 
$k=0,1,\ldots, p^{\lfloor m/2 \rfloor} -1$, and consequently~\eqref{eq:quad-cong2} has $p^{\lfloor m/2 \rfloor}$ solutions.

(iii) Since $p|D$, but $p^m\not|D$, there exists $i$ with $1 \leq i <m $ such that 
$D\equiv 0 \mod p^i$, but $D\not\equiv 0 \mod p^{i+1}$. Since  $\gcd(D/p^i,p)=1$, we infer that $p^i|Z^2$, but 
$p^{i+1} \not| Z^2$, where $Z$ is a solution of~\eqref{eq:quad-cong2}. It immediately follows that $i$ is even. 

We now rewrite~\eqref{eq:quad-cong2} as 
$$z^2 \equiv D/p^i \!\!\!\! \pmod{p^{m-i}}.$$
Since this congruence has a solution and $\gcd(D/p^i,p)=1$, by the same argument as in (i) above we conclude there are exactly two integers $k_1,k_2$, with $0 < k_1,k_2 < p^{m-i}$, such that
$$k_1^2 \equiv k_2^2 \equiv D/p^i \pmod{p^{m-i}}.$$
Thus the solutions of the congruence~\eqref{eq:quad-cong2} are of the form 
$$(k_1 + lp^{m-i})p^{i/2} \textrm{ or }(k_2 +lp^{m-i})p^{i/2}$$
with $l=0,1,\ldots, p^{i/2}-1$. Consequently,~\eqref{eq:quad-cong2} has $2p^{i/2}$ solutions.

In all three cases we see that~\eqref{eq:quad-cong2} has no more than $2p^{\lfloor m/2 \rfloor}$ solutions and so we obtain 
the bound~\eqref{eq:pt-line-up-bd}.
\end{proof}

Bound~\eqref{eq:pt-line-up-bd} shows that for $m \geq 3$, the points of $C_i$ are not 
collinear. We give a comprehensive proof of this fact to include the case $m=2$. 

\begin{lemma}\label{useful2}
Let $p >2$ and $m \ge 2$. For each $i$, with $i =1,2, \dots,p-1$ the points of $C_i$ are not collinear.
\end{lemma}

\begin{proof}
We argue by contradiction. Suppose there exists a line $L$ and an $i, 1 \leq i \leq p-1,$ such that 
$L\cap \cH_{p^m} = C_i$. By choosing the points on $C_i$ whose $x$-coordinates are $i$ and $i + p$ respectively, we infer that the slope of the line $L$ is an integer. The line $y=x$ is  a line of symmetry of $\cH_{p^m}$. If we reflect $L$ along $y=x$, we get a line $L^\prime$ such that 
$ L^\prime \cap \cH_{p^m} = C_{i^{-1} \!\!\mod p}$.  By the same argument as before we get that the slope of $L^\prime$ is an integer. Furthermore 
$ \textrm{slope}(L) \cdot \textrm{slope}(L^\prime) =1$ and consequently $\textrm{slope}(L) = \pm 1$. 

Suppose $\textrm{slope}(L)=-1$. Since $x=y$ is a line of symmetry of $\cH_{p^m}$, the reflection of $L$ along $x=y$ is $L$ 
itself. Consequently if $(s,t) \in C_i \cap L$, then $(t,s) \in C_i \cap L$. Since the number of points in $C_i$ is odd 
it follows that there must be a point in $C_i \cap L$ lying on $x=y$. Now the only points of $\cH_{p^m}$ that lie on $x=y$
are $(1,1)$ and $(p^m-1,p^m-1)$. Furthermore the line of slope $-1$ passing through $(1,1)$ contains no other points of 
$\cH_{p^m}$. A similar observation holds for $(p^m-1,p^m-1)$. In either case we obtain a contradiction to our assumption 
that $C_i \subseteq L$.

So the last case to consider is $\textrm{slope}(L)=1$. Let $(0,b)$  be the $y$-intercept of $L$. We have two possible cases:
$ i+b \geq p$ or $i +b < p.$

If $i+b \geq p$, then the 
point $(i + p^m-p, i + b + p^m-p)$ does not belong to $\cH_{p^m} \cap L$, and consequently the point of $C_i$ with 
$x$-coordinate $i + p^m-p$ does not lie on $L$. If $i +b <p$, then $i\cdot(i+b) < p^2$. Since $i\cdot(i+b) \equiv 1 \pmod{p^m}$, it follows that $i\cdot(i+b) =1$; that is, 
$(i,i+b)=(1,1)$ and $L$ is the line $x=y$. As noted earlier, this line contains only two points of $\cH_{p^m}$, and so once again we obtain a contradiction to our assumption that $C_i \subseteq L$. 
\end{proof}

We are now in a position to prove Theorem~\ref{Main-result}. We will need the following results.
For small $p^m$ we will invoke the following weaker version of the Dirac-Motzkin conjecture proved by Csima and Sawyer~\cite{CS}.

\begin{theorem}\label{6-13result}
Suppose $P$ is a finite set of $n$ points in the plane, not all on a line and $n \not= 7$. Then $P$ spans at least $6n/13$ ordinary lines.
\end{theorem}

The Dirac-Motzkin conjecture states that the lower bound for the number of ordinary lines is $n/2$ for sufficiently large $n$. 
Green and Tao~\cite[Theorem~2.2]{GT} in a 2013 preprint on arxiv have confirmed a more precise version of this conjecture which implies the following result. 

\begin{theorem}\label{green_tao}
 Suppose $P$ is a finite set of $n$ points in the plane, not all on a line and $n$ is sufficiently large.
 Then $P$ spans at least $(3/4+o(1))n$ ordinary lines if $n$ is odd and at least $n/2$ ordinary lines if $n$ is even. 
\end{theorem}

\begin{proof}[Proof of Theorem~\ref{Main-result}]
We first consider the case $p > 2.$ If $(x_1,y_1) \in C_i$ and $(x_2,y_2) \in C_j$ with $i \not= j$, then Lemma~\ref{useful1} (iii) shows that the line through the points $(x_1,y_1)$ and 
$(x_2,y_2)$ is ordinary. There are $(p-2)(p-1)p^{2(m-1)}/2$ possible such pairs of points. Furthermore, since the points of $C_i$ do not all lie on a line, by Theorem~\ref{6-13result}
or Theorem~\ref{green_tao}, respectively, each $C_i$ gives rise to 
at least $c\left(p^m\right)p^{m-1}$ ordinary lines. From these observations we conclude that 
\begin{equation*}
N  \geq  p^{m-1}(p-1)\left(\frac{p^{m-1}(p-2)}{2} + c(p^m)\right),
\end{equation*}
where $c(n)$ is defined in Theorem~\ref{Main-result}. For $p=2$, we have $C_1=\cH_{2^m}$ and consequently, 
$N \geq c\left(2^m\right) \cdot \#C_1 = c\left(2^m\right)2^{m-1}$. 
This is the same as the RHS of~\eqref{eq:low-bd-ord-line} as the first term of the RHS of~\eqref{eq:low-bd-ord-line} is 0.
\end{proof}

We now give an application of Beck's theorem~\cite[Theorem 3.1]{JB} to obtain an estimate for the number of lines spanned by $C_i$ when $m \geq 3$. We first state Beck's theorem in its original version.

\begin{theorem}[Beck]
Let $P$ be a set of $n$ points in the plane. Then at least one of the following holds:
\begin{enumerate}
\item[(i)] There exists a line containing at least $n/100$ points of $P$.
\item[(ii)] For some positive constant $c$, there exist at least $c\cdot n^2$ distinct lines containing two or more points of $P$.
\end{enumerate}
\end{theorem}

\begin{cor}
If 
$$ p^{\lceil m/2 \rceil-1} > 200,$$
then the number of lines spanned by $C_i$ with $i=1,\dots, p-1,$ is at least $ c \cdot p^{2(m-1)}$,
where $c$ is the constant in Beck's theorem.
\end{cor}

\begin{proof}
We apply Beck's theorem with $P=C_i$. By~\eqref{eq:pt-line-up-bd} the first case of Beck's theorem does not hold. Hence $C_i$ spans at least $c \cdot p^{2(m-1)}$ lines.
\end{proof}

\section{Shparlinski's Question}

One natural family of questions about finite point sets involves the various sets of distances they can determine. See for example \cite{BMP} or \cite{GIS}. In his survey paper~\cite{IS} on the properties of 
${\mathcal H}_{a,n}$, Shparlinski 
raises such a question.

Let  $\cF_{a,n}$ denote the set of Euclidean distances from the origin to points on 
$\cH_{a,n}$, that is, 
$$\cF_{a,n}= \{ \sqrt{x^2 + y^2} \, : \, (x,y) \in \cH_{a,n} \}.$$
In~\cite{IS} Shparlinski presents a proof by the fourth author(AW) that 
$$\# \cF_{a,p} = \frac{p+ (a/p)}{2}, \quad p >2, $$
where $p$ is  prime, $\gcd(a,p) =1$, and $(\cdot/p)$ is the Legendre symbol. It is natural to ask whether there is a similar formula 
for the cardinality $\# \cF_{a,n}$ for general $n$. The points of $\cH_{a,n}$ are symmetric along the line $y=x$ which suggests that 
$\# \cF_{a,n}$ is approximately $\varphi(n)/2$. The primary goal of this section is to adapt the proof in~\cite{IS} to estimate the difference 
$$\# \cF_{a,n} - \frac{\varphi(n)}{2}$$ when $n=p^2$ with $p$ an odd prime. 

To simplify the notation we introduce a map  $d_{a,n}: \Za_n \rightarrow \Z$ via 
$$d_{a,n}(x) = (x \bmod n)^2 + \left(\left(a \cdot x^{-1} \right) \bmod n \right)^2.$$
Clearly $\# {\rm Image}(d_{a,n}) = \# \cF_{a,n}$. 

We now focus on estimating $\# \Im$.

We should remark that determining the cardinality of the set 
$$\{ (x^2 + y^2) \!\! \mod n \, : \, (x,y) \in \cH_n \}$$
is easier and has been done in~\cite{HK} using completely elementary methods, that is, algebraic manipulations in conjunction with the Chinese Remainder Theorem.

\subsection{Some notation}

We begin by defining a class of biquadratic polynomials and certain subsets of $\Im$ and $\Za_{p^2}$. Let $f_u(Z)$ denote the polynomial 
$$f_u(Z)= Z^4-uZ^2+a^2.$$
Let $A \subseteq \Im$ be the set 
$$A= \{ u \in \Im : f_u(Z), f_u^{\prime}(Z) \textrm{ have no common root modulo } p\}$$
and let $B$ be the complement of $A$ in $\Im$. 

Let $B_1, B_2$ be the following two subsets of $\Im$. 
\begin{equation*}
B_1= \{ \dap(l) : l \in \Z_{p^2}^\ast, \,\,\, l^2 -a \equiv 0 \!\!\!\! \pmod{p} \},
\end{equation*}
and 
\begin{equation*}
B_2= \{ \dap(l) : l \in \Z_{p^2}^\ast, \,\,\, l^2 + a \equiv 0 \!\!\!\! \pmod{p} \}.
\end{equation*}

Finally if $a$ is a quadratic residue modulo $p$, then there is an integer $b, 0 < b < p$ such that $b^2 \equiv a \pmod{p}$. 
In this case we define the sets $C_1,C_2 \subseteq \Za_{p^2}$ via 
\begin{equation*}
C_1= \{ b+tp \, : \, 0\le t \le p-1 \},
\end{equation*}
\begin{equation*}
C_2= \{ p-b+tp \, : \, 0 \le t \le p-1 \}.
\end{equation*}

\subsection{Main result of Section 2 and proof} 

\begin{theorem}\label{sec-main-result}
Let $p>2$ be a prime. Then 
\begin{equation*}
\# \Im \, = \, \frac{\varphi(p^2)+1 +(a/p)}{2}-\#(d_{a,p^2}(C_1)\cap d_{a,p^2}(C_2)).
\end{equation*}
\end{theorem}

\noindent {\it Outline of proof of Theorem~\ref{sec-main-result}}. The proof is encapsulated in the following sequence of statements.
\begin{enumerate}
\item[(a)] We can associate each $u \in \Im$ with the congruence 
$$f_u(Z) \equiv 0 \!\!\!\! \pmod{p^2}.$$
\item[(b)] Using properties of $f_u(Z)$ we show that for each $u\in A$, there are exactly two distinct elements $x_1,x_2 \in \Za_{p^2}$ such that $$\dap(x_1) =\dap(x_2) =u.$$
\item[(c)] The cardinality of $A$ is 
$$\#A = \frac{\varphi(p^2)- \# \dap^{-1}(B)}{2}.$$
\item[(d)] The set $B$ is the disjoint union of the sets $B_1$ and $B_2$. Consequently,
$$\# \dap^{-1}(B)=\# \dap^{-1}(B_1)+\# \dap^{-1}(B_2).$$
\item[(e)] If $B_2\not= \emptyset$, then $\#\dap^{-1}(\{B_2\}) = 2p$ and $\#B_2 =p$.
\item[(f)]If $B_1 \not= \emptyset$, then $\#\dap^{-1}(\{B_1\})=2p$. Furthermore, $$ B_1 = \dap(C_1) \cup \dap(C_2)$$ with 
\begin{equation*}
\#\dap(C_i) = \frac{p-1}{2}+1,
\end{equation*} 
for $i=1,2$.
\end{enumerate}

\begin{proof}[Proof of (a),(b) and (c)]

Let $u \in \Im$. Then $u = r_u^2+ \left(ar_u^{-1}\right)^2$ for some $r_u \in \Z_{p^2}^\ast$ with $1 \leq r_u, ar_u^{-1} < p^2$.
It immediately follows that $r_u$ is a root of the congruence $f_u(Z) \equiv 0 \pmod{p}$.

We now turn to statements (b) and (c). Let $u \in A$ and let $r_u \in \Z_{p^2}^\ast$ such that $\dap(r_u)=u$. We claim that 
$$\dap^{-1}(\{u\}) = \{r_u, \, ar_u^{-1}\}.$$
We first show $r_u \not= ar_u^{-1}$, by proving the contrapositive. Let $x = r_u \mod p$ and 
$y = ar_u^{-1} \mod p$. If $r_u = ar_u^{-1}$, then $x=y$, $x^2 \equiv a \pmod{p}$ and $u \equiv 2 x^2 \pmod{p}$. It follows that $f_u(Z)$ factors as 
$$ f_u(Z) = Z^4-uZ^2 +a^2 \equiv (Z-x)^2(Z+x)^2 \!\!\!\! \pmod{p}.$$
But this contradicts our assumption that $f_u(Z)$ and $f^{\prime}_u(Z)$ do not have any roots in common modulo $p$. In a similar fashion we show that 
$ar_u^{-1} \not= p^2 -r_u.$

We now observe that $f_u(Z)$ has four distinct roots modulo $p$: $x, y, p-x$ and $p-y$. Furthermore each root lifts to a \emph{unique} root modulo $p^2$, that is, $x$ lifts to $r_u$, $y$ to $ar_u^{-1}$, $p-x$ to $(p^2-r_u)$ and $p-y$ to $(p^2-ar_u^{-1})$. Consequently 
$\dap^{-1}(\{u\}) \subseteq \{r_u, \, ar_u^{-1},\, p^2-r_u, \, p^2-ar_u^{-1} \}.$ So to conclude the proof we need to prove 
that $\dap(r_u) \not= \dap(p^2 -r_u)$. If $\dap(r_u) = \dap(p^2 -r_u)$, then a simple calculation shows $a r_u^{-1} = (p^2-r_u)$ which contradicts our earlier calculation that $ar_u^{-1} \not= p^2 -r_u.$
\end{proof}

\begin{proof}[Proof of (d)]
Let $\dap(r_u) =u$, where $u \in (\Im \cap B)$ and let $x = r_u \mod p$. Since $u \in B$, $x$ is a common root modulo $p$ of the polynomials 
$f_u(Z)= Z^4-uZ^2 +a^2$ and $f^{\prime}_u(Z)= 4Z^3-2uZ$. It follows that $2x^2 = u \pmod{p}$ and 
$$ (a-x^2)(a+x^2) \equiv 0 \!\!\!\! \pmod{p}.$$
Therefore $$x^2 \equiv a \!\!\!\! \pmod{p} \textrm{ and }u \equiv 2a \!\!\!\! \pmod{p}$$
or 
$$x^2 \equiv -a \!\!\!\! \pmod{p} \textrm{ and }u \equiv -2a \!\!\!\! \pmod{p}.$$
In the first case $u \in B_1$, and in the second $u \in B_2$. Finally  $B_1 \cap B_2 = \emptyset$ since $2a \not\equiv -2a \pmod{p}$. 
\end{proof}

\begin{proof}[Proof of (e)]
If $B_2 \not= \emptyset$, then there exists an integer $c$ with $1 \le  c \le p-1,$ such that $c^2  \equiv -a  \pmod{p}$. It follows that 
$d_{a,p^2}^{-1}(B_2)$ is the disjoint union of the sets $D_1,D_2$ where 
\begin{equation*}
D_1= \{ c+tp \, : \,  0 \le t \le p-1 \},
\end{equation*}
\begin{equation*}
D_2= \{ p-c+tp \; : \; 0 \le t \le p-1 \}.
\end{equation*}
Consequently, $\# \dap^{-1}(\{B_2\}) =2p$.

Now there exists a unique integer $l_p$, $ 0 \le l_p \le p-1$, such that 
$$c \cdot (p-c +l_pp) \equiv a \!\!\!\! \pmod{p^2}.$$
It follows that for $t =0,1,\dots,p-1$, 
\begin{equation*}
\left(a \cdot(c + tp)^{-1}\right) \!\!\!\! \mod p^2 = \left\{ \begin{array}{ll} p- c + (\lp + t)p, & \lp + t < p  \\ p- c + (\lp + t - p)p, & \lp + t \ge p. \end{array} \right. 
\end{equation*}
From this we see that $x \in D_1$ if and only if  $a \cdot x^{-1} \in D_2$, and we can conclude that the sets $d_{a,p^2}(D_1)$ and $d_{a,p^2}(D_2)$ 
are equal, and consequently $B_2 = d_{a,p^2}(D_1)$. So we are done if we can show that $d_{a,p^2}$ is one-to-one on $D_1$. To do this we 
define the functions 
$$f(t) = (c+tp)^2 + (p-c+(l_p+t)p)^2$$ 
and
$$g(t) = (c+tp)^2 + (p-c+(l_p+t-p)p)^2.$$
That is,
\begin{equation*}
d_{a,p^2}(c+tp) = \left\{ \begin{array}{ll} f(t), & \lp + t < p,  \\ g(t), & \lp + t \ge p. \end{array} \right. 
\end{equation*}
A simple calculation shows that $f(t) = f(s)$ if and only if $s=t$. Similarly, $g(t)=g(s)$ if and only if $s=t$. Finally, if we try to solve the 
equation $f(t)=g(s)$, we get the contradiction that $2|p$. Thus we get that $d_{a,p^2}$ is one-to-one on $D_1$.
\end{proof}

\begin{proof}[Proof of (f)]
If $B_1 \not= \emptyset$, then there exists an integer $b$ with $1 \le  b \le p-1,$ such that $b^2  \equiv a  \pmod{p}$.  
It follows that $d_{a,p^2}^{-1}(B_1)$ is the disjoint union of the sets $C_1,C_2$, where (we remind the reader)
$$ C_1= \{ b+tp \, : \, 0\le t \le p-1 \}, \textrm{ and } 
C_2= \{ p-b+tp \, : \, 0 \le t \le p-1 \}.$$
Consequently, $\# \dap^{-1}(\{B_1\}) =2p$.

The remaining part of the proof is trickier than the case for $B_2$. This is because $d_{a,p^2}$ is not one-to-one on $C_1$, nor are $d_{a,p^2}(C_1)$ and $d_{a,p^2}(C_2)$ equal as sets. We will prove that 
$$\#d_{a,p^2}(C_1)= \#d_{a,p^2}(C_2) = \frac{p-1}{2} +1.$$

Now there exists a unique integer $j_p$, $ 0 \le j_p \le p-1$, such that 
$$b \cdot (b +j_pp) \equiv a \!\!\!\! \pmod{p^2}.$$
It follows that for $t =0,1,\dots,p-1$, 
\begin{equation*}
\left(a \cdot(b + tp)^{-1}\right) \!\!\!\! \mod p^2 = \left\{ \begin{array}{ll} b + (j_p - t)p, &  t \le j_p  \\ b + (p+j_p - t)p, &  t > j_p. \end{array} \right. 
\end{equation*}
We now define the functions 
$$f(t) = (b+tp)^2 + (b+(j_p-t)p)^2 \textrm{ and } g(t) = (b+tp)^2 + (b+(p+j_p-t)p)^2.$$
That is,
\begin{equation}\label{eq:inv1}
d_{a,p^2}(b+tp) = \left\{ \begin{array}{ll} f(t), &  t \le j_p  \\ g(t), & t > j_p. \end{array} \right. 
\end{equation}
A simple calculation shows that $f(t) = f(s)$ if and only if $s=t$ or $s = j_p -t.$ Similarly, $g(t)=g(s)$ if and only if $s=t$ or $s=p +j_p -t$. 
Finally if we try to solve the equation $f(t)=g(s)$ we get the contradiction that $2|p$. These observations combined with the observation that either 
$(b + j_pp/2)$ or $(b+(j_p+p)p/2)$ is a solution of $x^2 \equiv a \pmod{p^2}$, give us the following:
\begin{enumerate}
\item[(i)] If $j_p$ is even, then $\#f^{-1}(\{t\})= 2$ for $ t \leq j_p, t \not= j_p/2$; $\#g^{-1}(\{t\})= 2$ for $ t > j_p$; and 
$\#f^{-1}(\{j_p/2\})= 1$.
\item[(ii)] If $j_p$ is odd, then $\#f^{-1}(\{t\})= 2$ for $ t \leq j_p$; $\#g^{-1}(\{t\})= 2$ for $ t > j_p, t \not= (j_p+p)/2$; and 
$\#f^{-1}(\{(j_p+p)/2\})= 1$.
\end{enumerate}
We conclude that 
$$\#d_{a,p^2}(C_1) = \frac{p-1}{2}+1.$$
In a similar manner we show that $\#d_{a,p^2}(C_2) = (p-1)/2+1.$  

In summary we see that if $(a/p)=1$, then 
$$\#B_1 = p+1 - \#\left(d_{a,p^2}(C_1)\cap d_{a,p^2}(C_2)\right).$$

\end{proof}

\subsection{Bounding $\#\left(d_{a,p^2}(C_1)\cap d_{a,p^2}(C_2)\right)$}

Thus the key difficulty to determining the cardinality $\# \Im$ is determining the cardinality of the intersection 
$d_{a,p^2}(C_1)\cap d_{a,p^2}(C_2).$ We now identify $C_1\times C_2$ with the set $ \{0,1,\dots,p-1\}^2$
via
$$(t,s) \mapsto (b + tp, p-b+sp)$$
and then define the map 
$$ l: \{0,1,\dots,p-1\}^2 \rightarrow \Z^2$$
via
$$l\left((t,s)\right) = (d_{a,p^2}(b+tp),d_{a,p^2}(p-b+sp)).$$ 
Clearly,  
$$\# \left(d_{a,p^2}(C_1)\cap d_{a,p^2}(C_2) \right) = \# \left( l\left([0,p-1]^2\right) \cap \{(x,x) \, : \, x \in \Z \} \right). $$ 

In~\eqref{eq:inv1} we gave the form of $(a\cdot x^{-1}) \mod p^2$ when $x\in C_1$, and then obtained the distance function associated 
with $C_1$. Specifically  
$$d_{a,p^2}(b+tp) = \left\{ \begin{array}{ll} f(t), &  t \le j_p  \\ g(t), & t > j_p \end{array} \right. $$
where 
$$f(t) = (b+tp)^2 + (b+(j_p-t)p)^2, \textrm{ and } g(t) = (b+tp)^2 + (b+(p+j_p-t)p)^2.$$

We now state a similar form when $x\in C_2$. Put
\begin{equation*}
k_p = \left\{ \begin{array}{ll} p-j_p-2, & j_p \leq p-2, \\
-1, & j_p = p-1. \end{array} \right.
\end{equation*}
Since $x \in C_2$,  
$x=p-b+sp$ for some $s$ with $0\le s \le p-1$.  An immediate calculation gives us the following:
\begin{equation*}
\left(a \cdot x^{-1}\right) \!\!\!\! \mod p^2 = \left\{ \begin{array}{ll} p-b + (k_p - s)p, &  s \le k_p,  \\ 
p-b + (p+k_p-s)p, &  s > k_p. \end{array} \right. 
\end{equation*}

Put $$F(s) = (p-b+sp)^2 + (p-b+(k_p-s)p)^2,$$
and 
$$G(s) = (p-b+sp)^2 + (p-b+(p+k_p-s)p)^2. $$
Then we have
\begin{equation*}
\dap(p-b+sp) = \left\{ \begin{array}{ll} F(s), &  s \le k_p,  \\ 
G(s), &  s > k_p. \end{array} \right. 
\end{equation*}

\begin{proposition}
Let $L_1,L_2$ be the sets 
\begin{eqnarray*}
 L_1 & = & \{ (t,s) \in [0,j_p/2] \times [k_p +1,(p+k_p)/2] \cap \Z^2 \, :  \\
& & \,\,\,\, (s+t+1-p)(s-t+1+j_p-p)= 2b+j_pp-p^2 \},
\end{eqnarray*}
\begin{eqnarray*} L_2 & = & \{ (t,s) \in [j_p+1,(p+j_p)/2] \times [0,k_p/2] \cap \Z^2 \, : \, \\
& & \,\,\,\, (s+t+1-p)(s-t+1+j_p)= 2b+j_pp\}.
\end{eqnarray*}
Then for $i =1,2$, if $L_i \not= \emptyset$, then $l$ is injective on $L_i$. Furthermore, 
\begin{equation*}
l\left([0,p-1]^2\right) \cap \{(x,x) \, : \, x \in \Z \} = l(L_1) \cup l(L_2). 
\end{equation*}
\end{proposition}

\begin{proof}

Let $(t,s)\in [0,p-1]^2\cap \Z^2$ such that $d_{a,p^2}(b+tp)=d_{a,p^2}(p-b+sp).$ We consider two cases: 
(a) $j_p \leq p-2$; (b) $j_p = p-1$.

Case (a) $j_p \leq p-2$. In this case we are forced to consider four equations:
\begin{enumerate}
\item[(i)] $f(t)- F(s)=0$: This has no solutions for integral $s$ and $t$. (Otherwise we get the contradiction $2|p$.)
\item[(ii)] $g(t)-G(s)=0$: Again this has no integer solutions for the same reason as above.
\item[(iii)] $f(t)-G(s) =0$: We have that $f(t)-G(s)$ equals the expression 
$$2p^2(-2p^2+2sp+2pj_p+2p-sj_p-tj_p-1+t^2-j_p+2b-2s-s^2).$$
Consequently $f(t)-G(s)=0$ simplifies to 
$$p^2-2sp-pj_p-2p+sj_p+tj_p+1-t^2+j_p+2s+s^2 = 2b + j_pp-p^2.$$
The LHS now factors to give 
\begin{equation}\label{eq:fac1}
(s+t+1-p)(s-t+1+j_p-p) = 2b+j_pp -p^2.
\end{equation}
\item[(iv)] $g(t)-F(s) =0$: We have that $g(t)-F(s)$ equals the expression
$$2p^2(2pj_p-tp+sp+p-sj_p-tj_p-1+t^2-j_p+2b-2s-s^2).$$ Consequently $g(t)-F(s)=0$ simplifies to 
$$-pj_p+tp-sp-p+sj_p+tj_p+1-t^2+j_p+2s+s^2= 2b+j_pp.$$
The LHS now factors to give 
\begin{equation}\label{eq:fac2}
(s+t+1-p)(s-t+1+j_p) = 2b+j_pp.
\end{equation}
\end{enumerate}

Case (b) $j_p = p-1$. In this case we consider the equation $f(t)-G(s)=0.$ We have that 
$$f(t)-G(s) = 2p^2(sp-tp-p+t^2+t-s^2+2b-s).$$
Consequently $f(t)-G(s)=0$ simplifies to 
$$(-sp+tp-t^2-t+s^2+s)= 2b-p.$$
The LHS factors to give
\begin{equation*} 
(s-t)(s+t+1-p) = 2b-p,
\end{equation*}
which we note is the same as~\eqref{eq:fac1} with $j_p=p-1$.

Thus we have proved that $(t,s)$ satisfies either~\eqref{eq:fac1} or~\eqref{eq:fac2}. Furthermore, it is easy to check that any point 
$(t,s) \in [0,p-1]^2\cap \Z^2$ satisfying either~\eqref{eq:fac1} or~\eqref{eq:fac2} must give that $d_{a,p^2}(b+tp)=d_{a,p^2}(p-b+sp).$ Thus to complete the proof we need to restrict ourselves to sets where $l$ is injective. 

We now note the following: 
\begin{enumerate}
\item[(I)] $f(t_2)=f(t_1)$ if and only if $t_2 = j_p-t_1$. 
\item[(II)] $G(s_2)=G(s_1)$ if and only if $s_2=p-k_p-s_1$. 
\item[(III)] $g(t_2)=g(t_1)$ if and only if $t_2= p+j_p-t_1$.
\item[(IV)] $F(s_2)= F(s_1)$ if and only if $s_2= k_p-s_1$.
\end{enumerate}

The condition for equation~\eqref{eq:fac1} arose when we considered the equation $f(t)=G(s)$. If we restrict ourselves to values of $t$ and $s$ 
satisfying this equation to the intervals $0 \leq t \leq j_p/2$, $k_p+1 \leq s \leq (p+k_p)/2$, we get that $l$ is injective. The condition for equation~\eqref{eq:fac2} arose when we considered the equation $g(t)=F(s)$. If we restrict ourselves to values of $t$ and $s$ 
satisfying this equation to the intervals $j_p+1 \leq t \leq (p+j_p)/2$, $0 \leq s \leq k_p/2$, we get that $l$ is injective. We conclude that 
\begin{equation*}
l\left([0,p-1]^2\right) \cap \{(x,x) \, : \, x \in \Z \} = l(L_1) \cup l(L_2) 
\end{equation*}
and consequently 
\begin{equation*}
\#\left(l\left([0,p-1]^2\right) \cap \{(x,x) \, : \, x \in \Z \}\right) = \# l(L_1) + \# l(L_2). 
\end{equation*}
\end{proof}

The interesting case of the previous proposition is the case for $j_p =0$. 
By setting $m = (s+t+1-p)$ and $n= (s-t+1)$, and then manipulating various inequalities we obtain the following corollary.
\begin{cor}\label{jequals0}
Let $j_p =0$ and let $S$ denote the set of lattice points 
 $(m,n) \in \Z^2$ with 
$mn = 2b$ satisfying the additional conditions:
$$ -p+2 \leq m < 0, \,\, -p/2 +1 \leq  n < 0, \, m \not\equiv n \!\!\!\! \pmod{2}, \, m \leq n.$$

Then $$\#S = \#l(L_2) = \#\left( d_{a,p^2}(C_1)\cap d_{a,p^2}(C_2) \right).$$
\end{cor}

We now have the following two corollaries.

\begin{cor}
For $p \geq 5$, $\#\left( d_{1,p^2}(C_1)\cap d_{1,p^2}(C_2) \right) =1$; consequently $\#{\rm Image}(d_{1,p^2}) = \varphi(p^2)/2$.
\end{cor}
\begin{proof}
Since $a=1$, we have $j_p =0$. Invoking Corollary~\ref{jequals0} we get that $S= \{(-2,-1)\}.$
\end{proof}

\begin{cor}
Let 
$$M_{p^2} = \max\left( \left\{ \, \varphi(p^2)/2 - \#\Im  \, : \, 1 \leq a < p^2, \, \gcd(a,p)=1 \right\}\right).$$ Then 
$$\lim_{p \rightarrow \infty} (M_{p^2}) = \infty.$$
\end{cor}

\begin{proof}
Let $a =p_1^2p_2^2\ldots p_n^2$, where $p_i$ is the $i$-th odd prime, and let $p$ be a prime larger than $a$.  Now $b = p_1p_2\ldots p_n$ and $j_p =0$, 
and therefore we can apply Corollary~\ref{jequals0}. The cardinality of $S$ (the set defined in Corollary~\ref{jequals0}) equals 
$2^n$. We now let $p$ and $n$ go to infinity to obtain our conclusion.
\end{proof}

\subsection{The case $n=p^m$, $m\geq 3$}

The reader should note for $p^m$ with $m \geq 3$, the proofs of statements (a),(b),(c) and (d) extend automatically. 
The higher power case starts to diverge from our earlier work when we start to consider the counterparts of the sets $B_1$ and $B_2$, which we denote as $B_{1,p^m}, B_{2,p^m}$, that is, 
\begin{equation*}
B_{1,p^m}= \{ d_{a,p^m}(l) : l \in \Z_{p^m}^\ast, \,\,\, l^2 -a \equiv 0 \!\!\!\! \pmod{p} \},
\end{equation*}
and 
\begin{equation*}
B_{2,p^m}= \{ d_{a,p^m}(l) : l \in \Z_{p^m}^\ast, \,\,\, l^2 + a \equiv 0 \!\!\!\! \pmod{p} \}.
\end{equation*}
The proofs that $\#\dap^{-1}(\{B_1\}) = 2p$ when $B_1 \not= \emptyset$, and  
$\#\dap^{-1}(\{B_2\}) = 2p$ when $B_2 \not= \emptyset$, extend to the general case. So we have the following. 
\begin{theorem}
For $i=1,2$, if $B_{i,p^m} \not= \emptyset$, then 
$$\#d_{a,p^m}^{-1}(B_{i,p^m}) = 2p^{m-1}.$$
Consequently,
\begin{eqnarray*}
\#{\rm Image}(d_{a,p^m})- \frac{\varphi(p^m)}{2}  & = & \left(\#B_{1,p^m}- \frac{(1+(a/p))p^{m-1}}{4} \right) \nonumber \\
& + &
\left(\#B_{2,p^m}-\frac{(1+(-a/p))p^{m-1}}{4} \right).
\end{eqnarray*}
In particular when $(a/p) = (-a/p) = -1$, and consequently $B_{1,p^m}= B_{2,p^m}= \emptyset$, then
\begin{equation}\label{eq:non-res-case}
\#{\rm Image}(d_{a,p^m})= \frac{\varphi(p^m)}{2}.
\end{equation}
\end{theorem}

Our final result provides a lower bound for $\# B_{i,p^m}, i=1,2,$ when  $B_{i,p^m}$ is non-empty. 

\begin{proposition}

If $u \in B_{i,p^m}$, then 
\begin{equation}\label{eq:circ-intersect}
 \#d_{a,p^m}^{-1}(\{u\}) \le 4p^{\lfloor m/2 \rfloor}.
\end{equation}
Consequently, if $B_{i,p^m}\ne \emptyset$ then 
\begin{equation}\label{eq:B-bound}
\#B_{i,p^m} \ge p^{\lceil m/2 \rceil -1}/2.
\end{equation}
\end{proposition}

\begin{proof}[Sketch] 
Without loss of generality let $i=1$. The proof of~\eqref{eq:circ-intersect} is a minor variation on the proof of~\eqref{eq:pt-line-up-bd}. Since $u\in B_{1,p^m}$, there exist $(x,y) \in \cH_{a,p^m}$ such that 
$$x^2 + y^2 =u.$$
We can transform the equation into the congruence
$$ z^2 \equiv (u^2-4a^2) \!\!\!\! \pmod{p^m},$$
where $z=(2x^2-u)$. We now copy the proof of Proposition~\ref{ptline-upperbound} to obtain~\eqref{eq:circ-intersect}. The only difference is that for each value of $z$ there are two possible values of $x$. Since $\#d_{a,p^m}^{-1}\left(B_{1,p^m}\right) = 2p^{m-1}$, \eqref{eq:B-bound} follows immediately.
\end{proof}

Our preliminary computations suggest that the bound~\eqref{eq:circ-intersect} is weak. Specifically, we have been unable to find a modular hyperbola for which there is a circle that intersects it at many points. Thus, unlike the case for $B_1$ and $B_2$, we are not satisfied with the bound for $B_{i,p^m}$.

\subsection{Some computed values of $\# \mathcal F_{a,p^m}$}
We conclude with the following tables of some small values of $\# \mathcal F_{a,p^m}$ computed directly. We point out that  the lines corresponding to $\# \mathcal F_{2,5^m}$ 
and $\# \mathcal F_{3,5^m}$ are redundant. This is because 
$(2/5)=(3/5)=-1$ and so we can simply invoke~\eqref{eq:non-res-case}.

\vskip .25in
\begin{tiny}

\noindent\begin{tabular}{l|llllllllll}
$m$ & 1 &  2 & 3 & 4 & 5 & 6 & 7 & 8 & 9 & 10\\
\hline
$\phi(3^m)/2$   & 1 &  3 & 9 & 27 & 81 & 243 & 729 & 2187 & 6561 & 19683 \\
$\# \mathcal F_{1,3^m}$  & 2 &  4 & 10 & 26 & 81 & 243 & 728 & 2185 & 6560 & 19682\\
$\# \mathcal F_{2,3^m}$  & 1 &  3 & 9 & 27 & 81 & 243 & 729 & 2187 & 6561 & 19683\\
$\# \mathcal F_{4,3^m}$  & 2 &  4 & 10 & 27 & 81 & 243 & 729 & 2185 & 6559 & 19681
\end{tabular}
\vskip.1in
\noindent\begin{tabular}{l|lllllll}
$m$ & 1 &  2 & 3 & 4 & 5 & 6 & 7 \\
\hline
$\phi(5^m)/2$   & 2 &  10 & 50 & 250 & 1250 & 6250 & 31250 \\
$\# \mathcal F_{1,5^m}$  & 3 & 10 & 51 & 249 & 1251 & 6248 & 31250 \\
$\# \mathcal F_{2,5^m}$  & 2 & 10 & 50 & 250 & 1250 & 6250 & 31250 \\
$\# \mathcal F_{3,5^m}$  & 2 & 10 & 50 & 250 & 1250 & 6250 & 31250 \\
$\# \mathcal F_{4,5^m}$  & 3 & 11 & 51 & 249 & 1251 & 6249 & 31248 
\end{tabular}
\vskip.1in
\noindent\begin{tabular}{l|lllllll}
$m$ & 1 &  2 & 3 & 4 & 5 & 6 & 7 \\
\hline
$\phi(7^m)/2$   & 3 &  21 & 147 & 1029 & 7203 & 50421 & 352947 \\
$\# \mathcal F_{1,7^m}$  & 4 & 21 & 148 & 1027 & 7203 & 50421 & 352946 \\
$\# \mathcal F_{2,7^m}$  & 4 & 22 & 147 & 1029 & 7204 & 50420 & 352943 \\
$\# \mathcal F_{3,7^m}$  & 3 & 21 & 147 & 1029 & 7203 & 50421 & 352947 \\
$\# \mathcal F_{4,7^m}$  & 4 & 21 & 148 & 1027 & 7204 & 50421 & 352946 
\end{tabular}
\vskip.1in

\end{tiny}

\end{document}